\documentclass[a4paper,11pt]{amsart}

\usepackage{nccmath} 
\usepackage{amssymb,amsthm,mathtools}

\usepackage{mathrsfs}

\usepackage{mathabx}

\usepackage{filecontents}

\usepackage{hyperref}

\usepackage{enumitem}
\usepackage[all]{xy}
\xyoption{line}
\usepackage{color}
\usepackage[normalem]{ulem}

\renewcommand{\le}{\varleq}
\renewcommand{\ge}{\vargeq}

\newcommand\mL{L\kern-0.08cm\char39}

\newcommand{\myforall}{\text{ for all }}

\newcommand{\myand}{\text{ and }}
\newcommand{\myif}{\text{ if }}

\newcommand{\mythen}{\text{ then }}

\newcommand{\seb}{\{\,}
\newcommand{\sen}{\,\}}

\newcommand{\pdirectional}{\raise0.05em\hbox{$+$}directional}
\newcommand{\pdirectionality}{\raise0.05em\hbox{$+$}directionality}
\newcommand{\pdirectionalitys}{\raise0.05em\hbox{$+$}directionality }
\newcommand{\pdirectionals}{\raise0.05em\hbox{$+$}directional }
\newcommand{\mdirectional}{\raise0.05em\hbox{$-$}directional}
\newcommand{\mdirectionality}{\raise0.05em\hbox{$-$}directionality}
\newcommand{\mdirectionalitys}{\raise0.05em\hbox{$-$}directionality }
\newcommand{\mdirectionals}{\raise0.05em\hbox{$-$}directional }

\newcommand{\Z}{\mathbb{Z}}
\newcommand{\Nonne}{\mathbb{N}}

\newcommand{\bi}{\in \Z}

\newcommand{\bpi}{\ge 1} 

\newcommand{\bni}{\ge 0} 

\newcommand{\diam}{{\rm diam}}

\newcommand{\dist}{{\rm dist}}

\newcommand{\ep}{\varepsilon}

\newcommand{\kuu}{\emptyset}
\newcommand{\nekuu}{\neq \kuu}
\newcommand{\iskuu}{= \kuu}

\newcommand{\bN}{\boldsymbol{N}}

\newcommand{\barV}{\bar{V}}

\newcommand{\hy}{\hat{y}}

\newcommand{\centb}{\begin{center}}
\newcommand{\centn}{\end{center}}

\newcommand{\enumb}{\begin{enumerate}}
\newcommand{\enumn}{\end{enumerate}}

\newcommand{\itemb}{\begin{itemize}}
\newcommand{\itemn}{\end{itemize}}

\numberwithin{equation}{section}

\usepackage{cleveref}
\setlist[enumerate,1]{label=(\alph*),ref=(\alph*)}
\setlist[enumerate,2]{label=(\arabic*),ref=(\alph{enumi}-\arabic{enumii})}
\setlist[enumerate,3]{label=(\Alph*),ref=(\roman{enumi}-\alph{enumii}-\Alph*)}
\setlist[enumerate,4]{label=(\arabic*),ref=(\roman{enumi}-\alph{enumii}-\Alph{enumiii}-\arabic*)}

\newlist{alphaenum}{enumerate}{1}
\setlist[alphaenum,1]{label=($\alpha$-\arabic*),ref=($\alpha$-\arabic*)}
\newlist{betaenum}{enumerate}{1}
\setlist[betaenum,1]{label=($\beta$-\arabic*),ref=($\beta$-\arabic*)}
\newlist{gammaenum}{enumerate}{1}
\setlist[gammaenum,1]{label=($\gamma$-\arabic*),ref=($\gamma$-\arabic*)}

\newtheorem{thm}{Theorem}[section]
\newtheorem{lem}[thm]{Lemma}
\newtheorem{prop}[thm]{Proposition}
\newtheorem{cor}[thm]{Corollary}

\theoremstyle{definition}
\newtheorem{defn}[thm]{Definition}

\theoremstyle{remark}

\crefname{sec}{\S}{\S\S}
\crefname{mainthm}{Theorem}{Theorems}
\crefname{thm}{Theorem}{Theorems}
\crefname{lem}{Lemma}{Lemmas}
\crefname{prop}{Proposition}{Propositions}
\crefname{cor}{Corollary}{Corollaries}
\crefname{defn}{Definition}{Definitions}
\crefname{conj}{Conjecture}{Conjectures}
\crefname{example}{Example}{Examples}
\crefname{nota}{Notation}{Notations}
\crefname{rem}{Remark}{Remarks}
\crefname{note}{Note}{Notes}
\crefname{case}{Case}{Cases}
\crefname{figure}{Figure}{Figures}
\crefname{section}{\S}{\S\S}
\crefname{enumi}{}{}
\crefname{enumii}{}{}
\crefname{equation}{}{}

\usepackage{url}

\begin{document}

\title[Finite-rank Bratteli--Vershik homeomorphisms are expansive]
{Finite-rank Bratteli--Vershik homeomorphisms are expansive}

\author{TAKASHI SHIMOMURA}

\address{Nagoya University of Economics,
 Uchikubo 61-1, Inuyama 484-8504, Japan}
\curraddr{}
\email{tkshimo@nagoya-ku.ac.jp}
\thanks{}



\date{\today}

\dedicatory{}

\commby{}

\begin{abstract}
Downarowicz and Maass (2008) have shown that every Cantor minimal homeomorphism
 with finite topological rank $K > 1$ is expansive.
Bezuglyi, Kwiatkowski, and Medynets (2009)
 extended the result to non-minimal aperiodic cases.
In this paper, we show that all finite-rank zero-dimensional
 systems are expansive or have infinite odometer systems;
this is an extension of the two aforementioned results.
Nevertheless, the methods follow similar approaches.
\end{abstract}

\maketitle

\section{Introduction}
Herman, Putnam, and Skau
 \cite{HERMAN_1992OrdBratteliDiagDimGroupTopDyn} have shown that
 a zero-dimensional system is essentially minimal if and only if
 it is represented as the Bratteli--Vershik system of
 an essentially simple ordered Bratteli diagram
 (see \cref{defn:essentially-simple}).
In \cite[Proposition 2.2]{Akin_2003LiYorkeSens},
 Akin and Kolyada completely characterized
 proximal topological dynamical systems.
According to them, a topological dynamical system
 $(X,f)$ is proximal if and only if it is essentially minimal and the
 unique minimal set is a fixed point.
Thus, zero-dimensional proximal systems have Bratteli--Vershik
 representations.
In \cite{DANILENKO_2001StrongOrbEquivLocalCompCantorMinSys,Matui_2002TopOrbEquivLocalCompCantorMinSys},
 Danilenko and Matui studied not only Cantor minimal systems
 but also locally compact Cantor minimal systems.
Locally compact Cantor minimal systems become proximal
 by one-point compactification.
Thus, it is necessary to study the quality of essentially minimal systems whose minimal sets
 are fixed points.
In this paper, a \textit{zero-dimensional system}
 implies a pair $(X,f)$
 of a compact zero-dimensional metrizable space $X$
 and a homeomorphism $f : X \to X$.
\textit{Odometer systems} are always infinite.
In this paper, we show that every finite-rank homeomorphic
 Bratteli--Vershik system without odometer systems is symbolic.
This is an elaborate task.
For zero-dimensional minimal systems,
 Downarowicz and Maass
 \cite{DOWNAROWICZ_2008FiniteRankBratteliVershikDiagAreExpansive}
  presented a remarkable theorem that states that
 every zero-dimensional minimal system of finite topological rank $K > 1$
 is expansive.
They used properly ordered Bratteli diagrams and adopted a noteworthy technique.
In \cite{BEZUGLYI_2009AperioSubstSysBraDiag},
 Bezuglyi, Kwiatkowski, and Medynets extended the aforementioned result
 to non-minimal cases that do not have periodic orbits.
Further, we showed that every zero-dimensional system has non-trivial
 Bratteli--Vershik representations
 (see \cref{thm:Bratteli-Vershik-for-all}
 or \cite{Shimomura_2016BratteliVershikrepForAllZeroDimHomeo}).
In this paper,
 we show that the symbolicity
 still holds for all finite-rank homeomorphic Bratteli--Vershik systems
 without odometers, in which periodic orbits may be allowed.
Our main result is as follows:
 if a zero-dimensional system has finite topological rank
 and no odometers,
 then it is expansive (see \cref{thm:main}).
The design of the proof presented in this paper
 essentially follows
 \cite{DOWNAROWICZ_2008FiniteRankBratteliVershikDiagAreExpansive}
 and the observation by Bezuglyi, Kwiatkowski, and Medynets
 \cite{BEZUGLYI_2009AperioSubstSysBraDiag}.
\section{Preliminaries}\label{sec:preliminaries}
Let $\Z$ denote the set of all integers,
 and let $\Nonne$ denote the set of all non-negative integers.
\begin{defn}
A \textit{Bratteli diagram} is an infinite directed graph $(V,E)$,
 where $V$ is the vertex set and $E$ is the edge set.
These sets are partitioned into non-empty disjoint finite sets
$V = V_0 \cup V_1 \cup V_2 \cup \dotsb$ and $E = E_1 \cup E_2 \cup \dotsb$,
 where $V_0 = \seb v_0 \sen$ is a one-point set.
Each $E_n$ is a set of edges from $V_{n-1}$ to $V_n$.
Therefore, there exist two maps $r,s : E \to V$ such that $r:E_n \to V_n$
 and $s : E_n \to V_{n-1}$ for all $n \bpi$,
i.e., the \textit{range map} and the \textit{source map}, respectively.
Moreover, $s^{-1}(v) \nekuu$ for all $v \in V$ and
$r^{-1}(v) \nekuu$ for all $v \in V \setminus V_0$.
We say that $u \in V_{n-1}$ is connected to $v \in V_{n}$ if there
 exists an edge $e \in E_n$ such that $s(e) = u$ and $r(e) = v$.
The \textit{rank $K$} of a Bratteli diagram is defined as
 $K := \liminf_{n \to \infty}\hash V_n$,
 where $\hash V_n$ is the number of elements in $V_n$.
\end{defn}
%
%
%
%
Let $(V,E)$ be a Bratteli diagram and $m < n$ be non-negative integers.
We define
\[E_{m,n} :=
 \seb p \mid
 p \text{ is a path from a } u \in V_m \text{ to a } v \in V_n \sen.\]
For a path $p = (e_{m+1},e_{m+2},\dotsc,e_n) \in E_{m,n}$, we define
 the source map $s(p) := s(e_{m+1})$ and the range map $r(p) := r(e_{n})$.
Then, we can construct a new Bratteli diagram $(V',E')$ as follows:
\[ V' := V_0 \cup V_1 \cup \dotsb \cup V_m \cup V_n \cup V_{n+1} \cup \dotsb \]
\[ E' :=
 E_1 \cup E_2 \cup \dotsb \cup E_m \cup E_{m,n} \cup E_{n+1} \cup \dotsb. \]
This procedure is called {\it telescoping}.
For a path $p = (e_{m+1},e_{m+2},\dotsc,e_n) \in E_{m,n}$
 and $m \le a < b \le n$,
 we denote $p|_{[a,b]} := (e_{a+1},e_{a+2},\dotsc,e_b)$.
For $n \bni$, we denote $E_{n,\infty}$ as the set of all infinite paths
 from $V_n$.
For a path $p \in E_{n,\infty}$ and $n \le a < b < \infty$,
 we define $p|_{[a,b]}$ as before,
 and we also define $p|_{[a,\infty)}$.
Let $0 \le m < n \le \infty$ and $m \le a \le n$.
For $v \in V_a$,
 we denote $E_{m,n}(v) := \seb p \in E_{m,n} \mid p \text{ passes } v \sen$.
Let $p \in E_{m,n}$.
We define a closed and open set
 $C(p) := \seb x \in E_{0,\infty} \mid x|_{[m,n]} = p \sen$, and
 call it a \textit{cylinder}.
\begin{defn}
Let $(V,E)$ be a Bratteli diagram such that
$V = V_0 \cup V_1 \cup V_2 \cup \dotsb$ and $E = E_1 \cup E_2 \cup \dotsb$
 are the partitions,
 where $V_0 = \seb v_0 \sen$ is a one-point set.
Let $r,s : E \to V$ be the range map and source map, respectively.
We say that $(V,E,\le)$ is an \textit{ordered} Bratteli diagram if
 the partial order $\le$ is defined on $E$ such that
 $e, e' \in E$ are comparable if and only if $r(e) = r(e')$.
 In other words, we have a linear order on each set
 $r^{-1}(v)$ with $v \in V \setminus V_0$.
The edges $r^{-1}(v)$ are numbered from $1$ to $\hash(r^{-1}(v))$.
Because $r^{-1}(v)$ is linearly ordered,
 we denote the maximal edge $e(v,\max) \in r^{-1}(v)$
 and the minimal edge $e(v,\min) \in r^{-1}(v)$.
\end{defn}
Let $n > 0$ and
 $e = (e_n,e_{n+1},e_{n+2},\dotsc), e'=(e'_n,e'_{n+1},e'_{n+2},\dotsc)$
 be cofinal paths from the vertices of $V_{n-1}$, which might be different.
We obtain the lexicographic order $e < e'$ as follows:
\[\myif k \ge n \text{ is the largest number such that }
 e_k \ne e'_k, \mythen e_k < e'_k.\]%
\begin{defn}
Let $(V,E,\le)$ be an ordered Bratteli diagram.
Let $E_{\max}$ and $E_{\min}$ denote the sets of maximal and minimal
edges, respectively.
An infinite (resp. finite) path is maximal (resp. minimal)
 if all the edges constituting the
path are elements of $E_{\max}$ (resp. $E_{\min}$).
\end{defn}
\begin{defn}\label{defn:max-min-path}
For $0 \le m < n < \infty$,
 we denote $E_{m,n,\max} := \seb p \mid
 p \in E_{m,n} \myand \myif p = (e_{m+1},e_{m+2},\dotsc,e_n),
 \mythen e_i \in E_{\max} \myforall i\ (m+1 \le i \le n) \sen$.
We also denote $E_{m,n,\min}$ similarly.
In the same manner, we denote $E_{m,\infty,\max}$ and $E_{m,\infty,\min}$.
\end{defn}
\begin{defn}\label{defn:essentially-simple}
As in \cite{HERMAN_1992OrdBratteliDiagDimGroupTopDyn},
 an ordered Bratteli diagram $(V,E)$ is called
 \textit{essentially simple}
 if the following exist: a unique infinite path
 $p_{\max} = (e_{\max,1},e_{\max,2},\dotsc)$
 with $e_{\max,i} \in E_{\max} \cap E_i$ for all $i \bpi$,
 and a unique infinite path
 $p_{\min} = (e_{\min,1},e_{\min,2},\dotsc)$
 with $e_{\min,i} \in E_{\min} \cap E_i$ for all $i \bpi$.
\end{defn}
\begin{defn}[Vershik map]
Let $(V,E,\le)$ be an ordered Bratteli diagram.
Let $E_{0,\infty}$ be endowed with
 the subspace topology of the product space $\prod_{i = 1}^{\infty}E_i$,
 with the discrete topology on each $E_i$ $(1 \le i < \infty)$.
Suppose that there exists a bijective map
 $\phi : E_{0,\infty,\max} \to E_{0,\infty,\min}$.
Then, we can define a map
 $\phi : E_{0,\infty} \to E_{0,\infty}$ as follows:

\noindent If $e = (e_1,e_2,\dotsc) \ne E_{0,\infty,\max}$,
 then there exists the least $n \ge 1$ such that
 $e_n$ is not maximal in $r^{-1}(r(e_n))$.
Then, we can select the least $f_n > e_n$ in $r^{-1}(r(e_n))$.
Let $v_{n-1} = s(f_n)$.
Then, it is easy to obtain the unique least path $(f_1,f_2,\dotsc,f_{n-1})$
 from $v_0$ to $v_{n-1}$.
We define
\[\phi(e) := (f_1,f_2,\dotsc,f_{n-1},f_n,e_{n+1},e_{n+2},\dotsc).\]
By definition, the map $\phi$ is bijective.
Suppose that the map $\phi : E_{0,\infty} \to E_{0,\infty}$ is a homeomorphism.
 Then, it is called a \textit{Vershik map}.
The system $(E_{0,\infty},\phi)$ is a zero-dimensional system
 (see \cite{HERMAN_1992OrdBratteliDiagDimGroupTopDyn}),
 and it is called the \textit{Bratteli--Vershik system} of $(V,E,\le)$.
\end{defn}
If a zero-dimensional system $(X,f)$ is topologically conjugate to
 a Bratteli--Vershik system $(E_{0,\infty},\phi)$,
 then it is called a Bratteli--Vershik representation of $(X,f)$.
For each vertex $v \in V_n$ with $n \bpi$,
 the set of cylinders $\seb C(p) \mid p \in E_{0,n}, r(p) = v \sen$ constitutes
 a \textit{tower} that is denoted as
 $T(v) := \bigcup_{p \in E_{0,n}, r(p) = v}C(p)$.
We showed the following in \cite{Shimomura_2016BratteliVershikrepForAllZeroDimHomeo}:
\begin{thm}\label{thm:Bratteli-Vershik-for-all}
Let $(X,f)$ be a zero-dimensional system, and
let $0 < l_1 < l_2 < \dotsb$ be an arbitrary infinite sequence of integers.
Then, $(X,f)$ has a Bratteli--Vershik representation with
 an ordered Bratteli diagram
 $( \seb V_n \sen_{n \bni}, \seb E_n \sen_{n \bpi}\sen)$
 such that, if
$\hash E_{0,n}(v) \le l_n$
 for some $v \in V_n$, then there exists a sequence
 $(v_n = v,v_{n+1},v_{n+2},\dotsc)$
 of vertices $v_{n+i} \in V_{n+i}$ $(i \bni)$
 such that there exists an $e_{n+i+1} \in E_{n+i+1}$ with
 $\seb e_{n+i+1} \sen = r^{-1}(v_{n+i+1})$
 and $s(e_{n+i+1}) = v_{n+i}$ for all $i \bni$.
\end{thm}
In \cite{DOWNAROWICZ_2008FiniteRankBratteliVershikDiagAreExpansive},
 Downarowicz and Maass introduced the topological rank for
 a Cantor minimal homeomorphism.
We define the same for all zero-dimensional systems.
\begin{defn}\label{defn:rank}
Let $(X,f)$ be a zero-dimensional system.
Then, the \textit{topological rank} of $(X,f)$ is $1 \le K \le \infty$
 if it has a Bratteli--Vershik representation with
 an ordered Bratteli diagram of rank $K$,
 and $K$ is the minimum of such numbers.
\end{defn}%
\noindent \textit{Question.}
In general, the topological rank in the context of all
 zero-dimensional systems is less than or equal to the original.
However, for Cantor minimal systems,
 the two topological ranks might be expected to coincide.

\vspace{3mm}

\section{Some Preparations and Related Results}
To prepare the proof of our main result,
 we essentially follow
 \cite{DOWNAROWICZ_2008FiniteRankBratteliVershikDiagAreExpansive}
 and the observation by Bezuglyi, Kwiatkowski, and Medynets \cite{BEZUGLYI_2009AperioSubstSysBraDiag}.
Let $(V,E,\ge)$ be an ordered Bratteli diagram.
For each $n \bpi$, we write $r_n := \hash V_n$
 and $V_n = \seb v_{n,1},v_{n,2},\dotsc, v_{n,r_n} \sen$.
We fix an $e = (e_1,e_2,\dotsc) \in E_{0,\infty}$.
For all $i \bi$,
 we denote $\phi^i(e) = (e_{1,i},e_{2,i},\dotsc)$.
Further, we denote $u_{n,i} = r(e_{n,i}) \in V_n$ $(n \bpi, i \bi)$.
For a vertex $v \in V_n$, let $l(v) := \hash E_{0,n}(v)$.
When $(e_{1,i},e_{2,i},\dotsc,e_{n,i})$ is minimal and $r(e_{n,i}) = v$,
 we change the symbol $u_{n,i} = v$ to $u_{n,i} = \check{v}$.
We write as $x_e := (u_{n,i})_{n \bpi, i \bi}$.
We define a shift map $(\sigma(x_e))_{n,i} = (u_{n,i+1})$
 $(n \bpi, i \bi)$.
Let
 $\barV_n
 := V_n \cup
 \seb \check{v}_{n,1}, \check{v}_{n,2}, \dotsc, \check{v}_{n,r_n} \sen$ with
 the discrete topology.
We consider the shift map
 $\sigma : \prod_{n \bpi}{\barV_n}^{\Z} \to \prod_{n \bpi}{\barV_n}^{\Z}$
 with the product topology.
There exists a dynamical embedding
 $\psi : (E_{0,\infty},\phi) \to (\prod_{n \bpi}{\barV_n}^{\Z},\sigma)$
 with $\psi(e) = x_e$.
We write $(X,f) := \psi(E_{0,\infty},\phi)$.
Fix an $n \bpi$ arbitrarily.
Following
 \cite{DOWNAROWICZ_2008FiniteRankBratteliVershikDiagAreExpansive},
 instead of using $\check{v}$,
 we make a \textit{cut} just before each occurrence of
 $\check{v}$ (see \cref{array-system}).
This coincides with the \textit{array system}
 (see \cite{DOWNAROWICZ_2008FiniteRankBratteliVershikDiagAreExpansive}).
Therefore, for each $x \in X$ and $n \bpi$,
 there exists a unique sequence
 $x[n]  := (\dotsc,u_{n,-2},u_{n,-1},u_{n,0},u_{n,1},\dotsc)$
 of vertices of $V_n$ that is separated by the cuts.
We make a convention $x[0] = (\dotsc,v_0,v_0,v_0,\dotsc)$
 that is cut everywhere.
We also write $x(n,i) := u_{n,i}$ for all $x \in X, n \bni$, and $i \bi$.
%
%
\begin{figure}
\begin{center}\leavevmode 
\xy
(0,18)*{}; (100,18)*{} **@{-},
 (48,16)*{v_{n,1}
 \hspace{6mm} v_{n,3} \hspace{6mm}
 \hspace{12mm} v_{n,1} \hspace{8mm}
 \hspace{5mm} v_{n,3} \hspace{5mm}
 \hspace{6mm} v_{n,2} \hspace{6mm}
 \hspace{4mm} v_{n,1}
 \hspace{2mm} },
(7,18)*{}; (7,14)*{} **@{-},
(28,18)*{}; (28,14)*{} **@{-},
(49,18)*{}; (49,14)*{} **@{-},
(70,18)*{}; (70,14)*{} **@{-},
(84,18)*{}; (84,14)*{} **@{-},
(0,14)*{}; (100,14)*{} **@{-},
 (55,12)*{v_{n+1,5}
 \hspace{36mm} v_{n+1,1} \hspace{6mm}
 \hspace{20mm} v_{n+1,3} \hspace{8mm} },
(28,14)*{}; (28,10)*{} **@{-},
(84,14)*{}; (84,10)*{} **@{-},
(0,10)*{}; (100,10)*{} **@{-},
\endxy
\end{center}
\caption{$n$th and $(n+1)$th rows of a linked array system with cuts.}
\label{array-system}
\end{figure}
%
%
%
%
%
For an interval $[n,m]$ with $m > n$, the combination of rows $x[n']$
 with $n \le n' \le m$ is denoted as $x[n,m]$.
The \textit{array system} of
 $x$ is the infinite combination $x[0,\infty)$
 of all rows $x[n]$ ($0 \le n < \infty$).
%
%
Note that for $m > n$, if there exists an $m$-cut at position $i$
 (just before position $i$), then
 there exists an $n$-cut at position $i$ (just before position $i$).
%
%
%
%
\begin{figure}
\begin{center}\leavevmode 
\xy
(0,22)*{}; (100,22)*{} **@{-},
 (50,20)*{v_0\phantom{{}_{{},f}}\ v_0\phantom{{}_{{},l}}
 \ v_0\phantom{{}_{{},l}}\ v_0\phantom{{}_{{},l}}
 \ v_0\phantom{{}_{{},f}}\ v_0\phantom{{}_{{},l}}
 \ v_0\phantom{{}_{{},l}}\ v_0\phantom{{}_{{},l}}
 \ v_0\phantom{{}_{{},f}}\ v_0\phantom{{}_{{},l}}
 \ v_0\phantom{{}_{{},l}}\ v_0\phantom{{}_{{},l}}
 \ v_0\phantom{{}_{{},f}}\ v_0\phantom{{}_{{},l}}},
(0,22)*{}; (0,18)*{} **@{-},
(7,22)*{}; (7,18)*{} **@{-},
(14,22)*{}; (14,18)*{} **@{-},
(21,22)*{}; (21,18)*{} **@{-},
(28,22)*{}; (28,18)*{} **@{-},
(35,22)*{}; (35,18)*{} **@{-},
(42,22)*{}; (42,18)*{} **@{-},
(49,22)*{}; (49,18)*{} **@{-},
(56,22)*{}; (56,18)*{} **@{-},
(63,22)*{}; (63,18)*{} **@{-},
(70,22)*{}; (70,18)*{} **@{-},
(77,22)*{}; (77,18)*{} **@{-},
(84,22)*{}; (84,18)*{} **@{-},
(91,22)*{}; (91,18)*{} **@{-},
(98,22)*{}; (98,18)*{} **@{-},
(0,18)*{}; (100,18)*{} **@{-},
 (49,16)*{c_{1,3}
 \hspace{8mm} v_{1,3} \hspace{6mm}
 \hspace{11mm} v_{1,1} \hspace{8mm}
 \hspace{7mm} v_{1,3} \hspace{5mm}
 \hspace{6mm} v_{1,2} \hspace{6mm}
 \hspace{4mm} v_{1,1}
 \hspace{2mm} },
(7,18)*{}; (7,14)*{} **@{-},
(28,18)*{}; (28,14)*{} **@{-},
(49,18)*{}; (49,14)*{} **@{-},
(70,18)*{}; (70,14)*{} **@{-},
(84,18)*{}; (84,14)*{} **@{-},
(0,14)*{}; (100,14)*{} **@{-},
(56,12)*{v_{2,3} \hspace{0.7mm}\
 \hspace{32mm} v_{2,1} \hspace{27mm}
 \hspace{10mm} v_{2,3}
 \hspace{2mm} },
(28,14)*{}; (28,10)*{} **@{-},
(84,14)*{}; (84,10)*{} **@{-},
(0,10)*{}; (100,10)*{} **@{-},
(36,8)*{c_{3,3} \hspace{26mm}
 \hspace{15mm} v_{3,1}},
(28,10)*{}; (28,6)*{} **@{-},
(0,6)*{}; (100,6)*{} **@{-},
(50,4)*{\vdots},
\endxy
\end{center}
\caption{The first 4 rows of an array system.}\label{array-system-2}
\end{figure}
%
%
%
%
%
For each vertex $v_{n,i}$, if we write

 $\seb e_1 < e_2 < \dotsb < e_{k(n,i)} \sen = r^{-1}(v_{n,i})$,
 we can determine a series of vertices
  $v_{n-1,a(n,i,1)} v_{n-1,a(n,i,2)} \dotsb v_{n-1,a(n,i,k(n,i))}$
 such that $v_{n-1,a(n,i,j)} = s(e_j)$ $(1 \le j \le k(n,i))$.
Furthermore,
 each $v_{n-1,a(n,i,j)}$ determines a series of vertices of level $n-2$
 similarly.
Thus, we can determine a set of vertices arranged in a square form as in
 \cref{2-symbol}.
Following \cite{DOWNAROWICZ_2008FiniteRankBratteliVershikDiagAreExpansive},
 this form is said to be the \textit{$n$-symbol} and denoted by $v_{n,i}$.
For $m < n$,
 the projection $v_{n,i}[m]$ that is a finite sequence of vertices of $V_m$
 is also defined.
\begin{figure}
\begin{center}\leavevmode 
\xy
(28,32)*{}; (84,32)*{} **@{-},
(57,30)*{v_0\phantom{{}_{{},l}}\ v_0\phantom{{}_{{},l}}
 \ v_0\phantom{{}_{{},l}}\ v_0\phantom{{}_{{},l}}
 \ v_0\phantom{{}_{{},f}}\ v_0\phantom{{}_{{},l}}
 \ v_0\phantom{{}_{{},l}}\ v_0\phantom{{}_{{},l}}},
(28,32)*{}; (28,28)*{} **@{-},
(35,32)*{}; (35,28)*{} **@{-},
(42,32)*{}; (42,28)*{} **@{-},
(49,32)*{}; (49,28)*{} **@{-},
(56,32)*{}; (56,28)*{} **@{-},
(63,32)*{}; (63,28)*{} **@{-},
(70,32)*{}; (70,28)*{} **@{-},
(77,32)*{}; (77,28)*{} **@{-},
(84,32)*{}; (84,28)*{} **@{-},
(28,28)*{}; (84,28)*{} **@{-},
 (60,26)*{
 \hspace{7mm} v_{1,1} \hspace{8mm}
 \hspace{7mm} v_{1,3} \hspace{6mm}
 \hspace{6mm} v_{1,2} \hspace{8mm}
 \hspace{2mm} },
(28,28)*{}; (28,24)*{} **@{-},
(49,28)*{}; (49,24)*{} **@{-},
(70,28)*{}; (70,24)*{} **@{-},
(84,28)*{}; (84,24)*{} **@{-},
(28,24)*{}; (84,24)*{} **@{-},
(60,22)*{
 \hspace{23mm} v_{2,1} \hspace{23mm}
 \hspace{2mm} },
(28,24)*{}; (28,20)*{} **@{-},
(84,24)*{}; (84,20)*{} **@{-},
(28,20)*{}; (84,20)*{} **@{-},
\endxy
\end{center}
\caption{The $2$-symbol corresponding to the vertex $v_{2,1}$
 of \cref{array-system-2}.}\label{2-symbol}
\end{figure}
%
%
%
%
%
%
It is clear that $x[n] = x'[n]$ implies that $x[0,n] = x'[0,n]$.
If $x \ne x'$ $(x, x' \in X)$,
 then there exists an $n > 0$ with $x[n] \ne x'[n]$.
For $x, x' \in X$,
 we say that the pair $(x,x')$ is \textit{$n$-compatible}
 if $x[n] = x'[n]$.
If $x[n] \ne x'[n]$,
 then we say that $x$ and $x'$ are \textit{$n$-separated}.
We recall that
 if there exists an $n$-cut at position $k$,
 then there exists an $m$-cut at position
 $k$ for all $0 \le m \le n$.
Let $x \ne x'$.
%
%
If a pair $(x,x')$ is $n$-compatible and $(n+1)$-separated,
 then we say that the \textit{depth of compatibility} of $x$ and $x'$ is $n$,
 or the pair $(x,x')$ has \textit{depth} $n$.
If $(x,x')$ is a pair of depth $n$
 and $(x,x'')$ is a pair of depth $m > n$,
 then the pair $(x',x'')$ has depth $n$ (hence, never equal).
%
%
An $n$-separated pair $(x,x')$ is said to \textit{have a common $n$-cut} if
 both $x$ and $x'$ have an $n$-cut at the same position.
If a pair has a common $n$-cut, then it also has a common $m$-cut
 for all $m$ ($0 \le m \le n$).
The set $X_n := \seb x[n] \mid x \in X \sen$ is a two-sided subshift
 of a finite set
 $\barV_n$.
 Just after the $n$-cuts, we have changed each symbol $v_{n,i}$ to
 $\check{v}_{n,i}$.
Thus, $X_n$ is a two-sided subshift of a finite set
 $\barV_n$.
The factoring map is denoted by $\pi_n : X \to X_n$,
 and the shift map is denoted by $\sigma_n : X_n \to X_n$.
We simply write $\sigma = \sigma_n$ for all $n$
 if there is no confusion.
Let $(Y,g)$ be a subsystem of $(E_{0,\infty},\phi)$, i.e.,
 $Y \subset E_{0,\infty}$ is closed,
 $\phi(Y) = Y$, and $g = \phi|_Y$.
Let $V' := \seb v \in V \mid T(v) \cap Y \nekuu \sen$.
We note that for all $p,p'$ with $r(p) = r(p) \in V_n$ ($n \bpi$),
 $C(p) \cap Y \nekuu$ if and only if $C(p') \cap Y \nekuu$.
Thus, for each $n > 0$,
 $p \in E_{0,n}$ satisfies $C(p) \cap Y \nekuu$
 if and only if $r(p) \in V'$.
Let $E' := \seb e \in E \mid r(e) \in V' \sen$.
Then, it is easy to check that $(V',E')$ is a Bratteli diagram.
We can give $(V',E')$ the order $\ge$ induced from the original.
All the maximal (resp. minimal) paths of $(V',E',\ge)$
 are maximal (resp. minimal) paths of $(V,E,\ge)$.
Let $Y_{\min}$ ($Y_{\max}$) be the set of all minimal (maximal)
 paths of $(V',E',\ge)$.
Then, it is obvious that $\phi(Y_{\min}) = Y_{\max}$.
Now, $(V',E',\ge)$ is an ordered Bratteli diagram, and the Bratteli--Vershik
 system that is topologically conjugate to $(Y,g)$ is defined.
Therefore, we get the following:
\begin{thm}\label{thm:subsystem}
Let $(X,f)$ be a zero-dimensional system of finite topological rank $K$.
Then, its subsystem has topological rank $\le K$.
\end{thm}
The next proposition follows \cite[Proposition 4.6]{BEZUGLYI_2009AperioSubstSysBraDiag}.
Nevertheless, because the proposition was previously described in the context of
 aperiodic systems, we present the proof here as well.
We fix a metric $d$ on $E_{0,\infty}$.
Let $M, M' \subseteq E_{0,\infty}$ be closed sets.
Then, we denote
\[\dist(M,M') := \min_{x \in M,\ y \in M'}d(x,y).\]
For a closed set $M \subseteq E_{0,\infty}$, we denote
 $\diam(M) := \max_{x,y \in M}d(x,y)$.
\begin{prop}\label{prop:finite-minimal-set}
Let $(X,f)$ be a zero-dimensional system of finite rank $K$.
Then, $(X,f)$ has at most $K$ minimal sets.
\end{prop}
\begin{proof}
Suppose that there exist $K+1$ minimal sets, namely $M_1,M_2,\dotsc,M_{K+1}$.
Take $\ep > 0$ such that
 $\ep < \min_{1 \le i < j \le K+1} \dist(M_i, M_j)$
and $n > 0$ such that $\max \seb \diam(C(p)) \mid p \in E_{0,n} \sen \le \ep$.
For each $i\ (1 \le i \le K+1)$,
 we define $P_i := \seb p \in E_{0,n} \mid C(p) \cap M_i \nekuu \sen$.
It is evident that $P_i \nekuu$ for all $i\ (1 \le i \le K+1)$.
It is also obvious that $P_i \cap P_j \iskuu$ for $i \ne j$.
We note that for each $v \in V_n$,
 $\seb C(p) \mid p \in E_{0,n}, r(p) = v \sen$
 is a tower that constructs $T(v)$.
Thus, for each $v \in V_n$ and $i\ (1 \le i \le K+1)$,
 only one of $\seb p \in E_{0,n} \mid r(p) = v \sen \subseteq P_i$
 or $\seb p \in E_{0,n} \mid r(p) = v \sen \cap P_i \iskuu$ occurs,
 which is a contradiction.
\end{proof}

\section{Main Theorem.}
%
%
Before the statement of our main theorem, we choose
 a small $\ep_0 > 0$.
Let $M_1, M_2, \dotsc, M_{K'}$ ($K' \le K$)
 be the list of all the minimal sets of $(E_{0,\infty},\phi)$.
Fix $\ep_0 > 0$ such that
 $\ep_0 < \frac{1}{3} \min \seb \dist(M_i,M_j) \mid 1 \le i < j \le K' \sen$.
Suppose that there exist $x \ne y \in E_{0,\infty}$ such that
 $\limsup_{n \to +\infty} d(f^n(x),f^n(y)) \le \ep_0$.
Then, it is easy to see that the minimal sets of $\omega(x)$ and $\omega(y)$
 coincide.
Thus, we get the next two lemmas:
\begin{lem}\label{lem:ep0-omega}
Let $x,y \in E_{0,\infty}$.
Then, we get the following:
\itemb
\item if $\limsup_{n \to +\infty}d(f^n(x),f^n(y)) \le \ep_0$,
 then the minimal sets of $\omega(x)$ and $\omega(y)$ coincide.
\itemn
\end{lem}
\begin{lem}\label{lem:ep0-alpha}
Let $x,y \in E_{0,\infty}$.
Then, we get the following:
\itemb
\item if $\limsup_{n \to -\infty}d(f^n(x),f^n(y)) \le \ep_0$,
 then the minimal sets of $\alpha(x)$ and $\alpha(y)$ coincide.
\itemn
\end{lem}
From the two lemmas stated above, we get the next lemma:
\begin{lem}\label{lem:ep0}
Let $x,y \in E_{0,\infty}$.
If $\sup_{n \bi}d(f^n(x),f^n(y)) \le \ep_0$,
 then the minimal sets of $\alpha(x)$ and $\alpha(y)$ coincide
 and the minimal sets of $\omega(x)$ and $\omega(y)$ coincide.
\end{lem}
\begin{thm}[Main Result]\label{thm:main}
Let $(X,f)$ be a finite-rank zero-dimensional system
 such that no minimal set is an odometer.
Then, $(X,f)$ is expansive.
\end{thm}
\begin{proof}
Let $K \ge 1$ be the topological rank of $(X,f)$.
We note that if $K = 1$, then we obtain an ordered Bratteli diagram
 that has rank $K = 1$.
Then, because $(X,f)$ is not an odometer,
 the Bratteli--Vershik system is a single periodic orbit
 and is expansive.
Thus, the conclusion of the theorem is obvious.
Therefore, we assume that $K > 1$.
As with the proof presented in
 \cite{DOWNAROWICZ_2008FiniteRankBratteliVershikDiagAreExpansive},
 we show our proof by contradiction.
Suppose that the claim fails.
Then, for all $L > 0$, there exists a pair $(x,x')$ with distinct elements
 of $X$ that is $L$-compatible.
Because $x \ne x'$, for some $m > L$, $(x,x')$ is $m$-separated.
Therefore, $(x,x')$ has depth $n$ with $L \le n < m$.
Therefore,
 for infinitely many $n$, there exists a pair $(x_n,x'_n)$ of depth $n$.
By telescoping, we can assume that
 every $n > 0$ has a pair $(x_n,x'_n)$ of depth $n$.
Note that even after another telescoping, this quality still holds.
As with the proof presented in
 \cite{DOWNAROWICZ_2008FiniteRankBratteliVershikDiagAreExpansive},
 we show our proof for separate cases:
\enumb
\renewcommand{\labelenumi}{(\arabic{enumi})}
\renewcommand{\theenumi}{(\arabic{enumi})}
\setcounter{enumi}{0}
\item\label{case1} there exists an $N$ such that
 for all $n > N$ and every $m > n$,
 there exists a pair $(x_n,x'_n)$ of depth $n$ with a common $m$-cut;
\item\label{case2} for infinitely many $n$, and every sufficiently large
 $m > n$, any pair of depth $n$ has no common $m$-cut.
\enumn

\vspace{3mm}

\noindent \textit{Proof for case \cref{case1}.}
In \cite{BEZUGLYI_2009AperioSubstSysBraDiag}, the proof is omitted.
However, we present it here for the readers' convenience.
As with the proof presented in
 \cite{DOWNAROWICZ_2008FiniteRankBratteliVershikDiagAreExpansive},
 we prove that such a case never occurs.
Fix some $m > N + K$, and for each integer $n \in [m - K, m -1]$,
 let $(x_n,x'_n)$ be
 a pair of depth $n$ with a common $m$-cut.
For $n = m -1$, we have an $(m - 1)$-compatible $m$-separated pair
 $(x_{m-1},x'_{m-1})$ with a common $m$-cut.
%
%
%
Suppose that all the $m$-cuts of $x_{m-1}$ and $x'_{m-1}$ are the same.
Because the pair $(x_{m-1},x'_{m-1})$ is $m$-separated,
 at least two distinct vertices are in the same place of
 $x_{m-1}[m]$ and $x'_{m-1}[m]$.
These symbols have the same rows from $0$ to $m-1$.
Therefore,
 at least two distinct $m$-symbols have the same rows from $0$ to $m-1$.
If the sets of $m$-symbols have the same rows from $0$ to $m-1$,
 we factor these to the same alphabet,
 i.e., we make a new ordered Bratteli diagram
 identifying such vertices of $V_m$.
Suppose that $v$ and $v'$ are identified to a single $v$.
Then, we assume that the sets $r^{-1}(v)$ and $r^{-1}(v')$
 are identified to a single $r^{-1}(v)$ with the same order.
In addition, we assume that $s^{-1}(v)$ and $s^{-1}(v')$ are joined to form the new
 $s^{-1}(v)$, and the orders are not changed.
Further, $E_{0,\infty}$ has a canonical isomorphism to the original,
 and the $X_{m-1}$ of the new Bratteli diagram is the same as the old one.
Because the pair $(x_{m-2},x'_{m-2})$ was $(m-1)$-separated, and
 this factorization does not affect the $(m-1)$th row,
 the pair $(x_{m-2},x'_{m-2})$ is still $m$-separated.
Note that the number of vertices of the $m$th row has been decreased by at least $1$.
Next, we consider the case in which after a common $m$-cut of
 the pair $(x_{m-1},x'_{m-1})$,
 the coincidence of the positions of the $m$-cuts
 does not continue toward the right or left end.
\begin{figure}
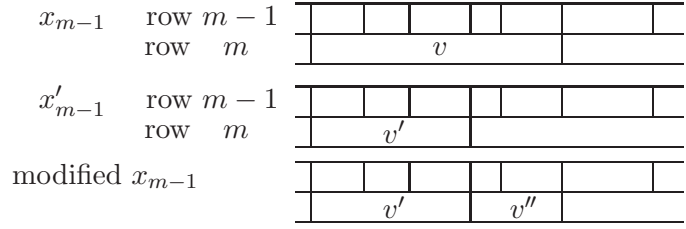

\begin{center}\leavevmode 
\xy
(33,28)*{}; (84,28)*{} **@{-},
(15,26)*{x_{m-1} \hspace{4mm} \text{ row } m-1},
(60,26)*{},
(35,28)*{}; (35,24)*{} **@{-},
(42,28)*{}; (42,24)*{} **@{-},
(48,28)*{}; (48,24)*{} **@{-},
(56,28)*{}; (56,24)*{} **@{-},
(60,28)*{}; (60,24)*{} **@{-},
(68,28)*{}; (68,24)*{} **@{-},
(80,28)*{}; (80,24)*{} **@{-},
%
(33,24)*{}; (84,24)*{} **@{-},
(20,22)*{\text{row \hspace{3mm}} m },
(60,22)*{
 \hspace{16mm} v \hspace{30mm}
 \hspace{2mm} },
(35,24)*{}; (35,20)*{} **@{-},
(68,24)*{}; (68,20)*{} **@{-},
%
(33,20)*{}; (84,20)*{} **@{-},
(33,17)*{}; (84,17)*{} **@{-},
(15,15)*{x'_{m-1} \hspace{4mm} \text{ row } m-1},
(60,13)*{},
(35,17)*{}; (35,13)*{} **@{-},
(42,17)*{}; (42,13)*{} **@{-},
(48,17)*{}; (48,13)*{} **@{-},
(56,17)*{}; (56,13)*{} **@{-},
(60,17)*{}; (60,13)*{} **@{-},
(68,17)*{}; (68,13)*{} **@{-},
(80,17)*{}; (80,13)*{} **@{-},
%
(33,13)*{}; (84,13)*{} **@{-},
(20,11)*{\text{row \hspace{3mm}} m },
(60,11)*{
 \hspace{10mm} v' \hspace{36mm}
 \hspace{2mm} },
(35,13)*{}; (35,9)*{} **@{-},
(56,13)*{}; (56,9)*{} **@{-},
%
(33,9)*{}; (84,9)*{} **@{-},
(33,7)*{}; (84,7)*{} **@{-},
(10,5)*{\text{modified } x_{m-1} \hspace{4mm}},
(60,3)*{},
(35,7)*{}; (35,3)*{} **@{-},
(42,7)*{}; (42,3)*{} **@{-},
(48,7)*{}; (48,3)*{} **@{-},
(56,7)*{}; (56,3)*{} **@{-},
(60,7)*{}; (60,3)*{} **@{-},
(68,7)*{}; (68,3)*{} **@{-},
(80,7)*{}; (80,3)*{} **@{-},
%
(33,3)*{}; (84,3)*{} **@{-},
%
(60,1)*{
 \hspace{10mm} v' \hspace{36mm}
 \hspace{2mm} },
(77,1)*{
 \hspace{10mm} v'' \hspace{36mm}
 \hspace{2mm} },
(35,3)*{}; (35,-1)*{} **@{-},
(56,3)*{}; (56,-1)*{} **@{-},
(68,3)*{}; (68,-1)*{} **@{-},
%
(33,-1)*{}; (84,-1)*{} **@{-},
\endxy
\end{center}
\caption{Change of an $m$-symbol. There is a common cut at the left end.}
\label{change-symbol}
\end{figure}
Suppose that after a common $m$-cut at position $k_0$,
 the common $m$-cuts do not continue to the right end.
Let $k_1 > k_0$ be the position of the first common $m$-cut
 such that the right $m$-symbols $v$ and $v'$ of $x_{m-1}$ and $x'_{m-1}$
 have different lengths.
We assume without loss of generality that $l(v') < l(v)$.
We recall that for any vertex $v \in V \setminus V_0$,
 $l(v) := \hash E_{0,n}(v)$.
Let $k_2 = k_1 + l(v')$.
Then, $k_2$ is the position of the next $m$-cut of $x'_{m-1}$.
Because $x_{m-1}[m-1] = x'_{m-1}[m-1]$,
 the $m$-symbol $v$ itself has an $(m-1)$-cut at $l(v')$ from the
 left.
Then, we separate $v$ into two parts at the position, making a new vertex
 $v'' \in V_m$ such that $v''[0,m-1]$ is the right half of $v[0,m-1]$.
The left half is identified with $v'$ (see \cref{change-symbol}).
We note that $v''$ is a ``new'' $m$-symbol, even if there has been
 an original symbol $v'''$ with $v'''[0,m-1] = v''[0,m-1]$.
We replace every occurrence of the $m$-symbol $v$ in every element of $X_m$
 by the concatenation $v'v''$.
In the Bratteli diagram, we delete $v$ and add $v''$,
 and we replace the edges that connect $v$ with the edges that connect $v'$ or $v''$.
In accordance with the new $(m+1)$-symbols,
 each edge in $E_{m+1}$ that had connected $v$ has to be duplicated into
 two edges, one connecting $v'$ and the other connecting $v''$.
Further, the orders of the edges of $E_{m+1}$ have to be changed.
After all the necessary changes, we show that
 the new ordered Bratteli diagram has the Bratteli--Vershik system
 that is canonically isomorphic to the original.
We note that the maximal or minimal paths are joined to the infinite cuts
 from level $0$ to $\infty$.
Because the $(m+1)$th row is not changed,
 no new infinite cut arises.
Therefore, no new infinite maximal path or infinite minimal path arises.
Thus, we need not change the Vershik map up to canonical isomorphism;
 the existence of canonical isomorphism is evident.
The number of vertices is not changed.
After this modification, no cut that existed is removed.
The coincidence of the $m$-cut from $k_0$ to the right might be shortened.
Nevertheless, the same modification is possible, and
 finally, we reach the point where we have the same $m$-cuts
 from $k_0$ to the right end.
The same argument is valid for the left direction from $k_0$.
Now,
 $x_{m-1}$ and $x'_{m-1}$ have the same $m$-cuts throughout the sequences.
Then, we can apply the previous argument.
Thus,
 we get a factor in the $m$th row that decreases the number of symbols
 $V_m$, and $(x_{m-2},x'_{m-2})$ is still $m$-separated.
We can now delete the $(m-1)$th row and continue this process.
Finally, the $m$th row is represented by only one vertex,
 and the pair $(x_{m-K},x'_{m-K})$
 is still $m$-separated and has a common $m$-cut,
 which is a contradiction.

\vspace{5mm}

%
%
%
\noindent \textit{Proof for case \cref{case2}}.
For an arbitrarily large $n$, there exists an $m(n)>n$ such that
 every pair $(x_n,x'_n)$ of depth $n$
 has no common $m(n)$-cut.
As described briefly in
 \cite{DOWNAROWICZ_2008FiniteRankBratteliVershikDiagAreExpansive},
 by telescoping,
 we wish to show that the condition of \cref{case2} holds for every $n$.
Fix $n$.
Take an $n' > m(n)$ such that
 every pair $(x_{n'},x'_{n'})$ of depth $n'$
 has no common $m(n')$-cut.
Then, because $n' > m(n)$, every pair $(x_n,x'_n)$ of depth $n$
 has no common $n'$-cut.
Thus, by telescoping (from $n'$ to $n$),
 every pair $(x_n,x'_n)$ of depth $n$ has
 no common $(n+1)$-cut.
Thus, through consecutive application of such telescoping,
 we get an ordered Bratteli diagram such that
 for every $n > 0$, every pair $(x_n,x'_n)$ of depth $n$
 has no common $(n+1)$-cut.
For each $n > 0$,
 let $(x_n,x'_n)$ be a pair of depth $n$ that has no common $(n+1)$-cut.
There exists an $N > 0$ such that every pair $(x_n,x'_n)$ of depth $n \ge N$
 satisfies $\sup_{i \bi}d(f^i(x_n),f^i(x'_n)) \le \ep_0$.
By \cref{lem:ep0}, for any pair $(x_n,x'_n)$ of depth $n \ge N$,
 the minimal sets of $\omega(x_n)$ and $\omega(x'_n)$ coincide
 and the minimal sets of $\alpha(x_n)$ and $\alpha(x'_n)$ coincide.
First, we claim that $\omega(x_n) \cap \omega(x'_n)$ with $n \ge N$
 contains a periodic point only for finitely many $n \ge N$.
Suppose that there exists a periodic point $y$ and an infinite set
 $\bN \subset \seb n \mid n \ge N \sen$ such that
 $y \in \omega(x_n) \cap \omega(x'_n)$ for all $n \in \bN$.
Let $Z_n$ be the positions of the $n$-cuts of $y$.
Then, $Z_i$ ($i \in \bN$) is periodic and $Z_n \supseteq Z_m$ for $n < m$.
Therefore, there exists an infinite set $\bN' \subset \bN$ such that
 for all $n \in \bN'$, the $n$-cuts of $y$ are the same.
Let $n \in \bN'$.
There exists a sequence $k(1) < k(2) < \dotsb$ with
 $\lim_{i \to +\infty}f^{k(i)}(x_n) = y$.
Taking a subsequence if necessary,
 we get $\lim_{i \to +\infty}f^{k(i)}(x'_n) = y' \in E_{0,\infty}$.
Because $x_n[n] = x'_n[n]$, we get $y[n] = y'[n]$.
Thus, $y$ and $y'$ have the same $n$-cuts.
Similarly, $y$ and $y'$ have no common $(n+1)$-cut.
Nevertheless, the positions of the $(n+1)$-cuts of $y$ are the same as the positions
 of the $n$-cuts of $y$,
 and the $(n+1)$-cuts of $y'$ have to be a part of the $n$-cuts of $y'$ and of $y$.
Thus, $y,y'$ have common $(n+1)$-cuts, which is a contradiction.
We have proved the claim.
Therefore, there exists an $N' > N$ such that
 for all $n \ge N'$, $\omega(x_n) \cap \omega(x'_n)$ contains no periodic
 point.
By telescoping, there exists a minimal set $M$ such that
 for all $n \bpi$, $M \subseteq \omega(x_n) \cap \omega(x'_n)$.
The remainder of the argument proceeds as that of the Infection lemma of
 \cite{DOWNAROWICZ_2008FiniteRankBratteliVershikDiagAreExpansive}.
Nevertheless, we cannot assume $M = E_{0,\infty}$ in general.
We only need to check that this fact does not cause any problem.
Let $i_0 > 0$ and $L > 0$ be arbitrarily large integers.
Let $j = i_0+L$.
For each $i \in [i_0,j-1]$, the pair $(x_i,x'_i)$ has depth $i$ with
 no common $(i+1)$-cuts.
Fix a $y_0 \in M$.
As in \cite{DOWNAROWICZ_2008FiniteRankBratteliVershikDiagAreExpansive},
 for each $i \in [i_0,j-1]$,
 by applying some element $\tau_i$ of the enveloping semi-group,
 we can get a pair $(\tau_i(x_i),\tau_i(x'_i))$ with $\tau_i(x_i) = y_0$.
By letting $y_i = \tau_i(x'_i)$ ($i \in [i_0,j-1]$), we get that
 each $(y_0,y_i)$ is $i$-compatible with no common $(i+1)$-cuts.
Let $i,i' \in [i_0,j-1]$ satisfy $i < i'$.
Then, $y_{i'}[i+1] = y_0[i+1]$, and $y_{i'}$ has no common $(i+1)$-cut with
 $y_i$.
Thus, we get finite elements $y_0$ and $y_i$ ($i \in [i_0,j-1]$).
These elements are all $i_0$-compatible and pairwise $j$-separated.
 Furthermore, they have no common $j$-cuts pairwise.
We have to take care that $y_i$ ($i \in [i_0,j-1]$) need not be elements of
 $M$.
Let $M(i) := \pi_i(M)$ for all $i \bni$.
The next lemma is called the Infection lemma
 in \cite{DOWNAROWICZ_2008FiniteRankBratteliVershikDiagAreExpansive}.

\vspace{2mm}

\noindent \textbf{Infection lemma}:
Suppose that there exist at least $K^{K+1} + 1$ $i$-compatible points
 $y_k \in E_{0,\infty}$
 ($k \in [1, K^{K+1} + 1]$), which, for some $j > i$, are pairwise
 $j$-separated with no common $j$-cuts.
Let $\hy = y_k[i] \in M(i)$ be the common sequence.
 Then, $\hy$ is eventually periodic.

\vspace{2mm}

The proof is factually identical to that of the original.
Thus, we omit the proof here.
By the Infection lemma, $M(i_0)$ has a periodic orbit.
Nevertheless, by the minimality of $M$, $M(i_0)$ is also minimal.
Thus, $M(i_0)$ is a periodic orbit.
Because $i_0$ can be arbitrarily large, it follows that
 $M$ is an odometer, which is a contradiction.
\end{proof}
We think that we have settled down the extension process of
 \cite[Theorem 1]{DOWNAROWICZ_2008FiniteRankBratteliVershikDiagAreExpansive}
 at least once.
As an example, we get the following:
\begin{cor}\label{cor:proximal-case}
If $(X,f)$ is a finite-rank proximal zero-dimensional system,
 then it is symbolic.
\end{cor}

\noindent
\textsc{Acknowledgments:}
The author would like to thank anonymous referee(s) who reviewed our previous submission.
Their comments encouraged us to write this paper.
This work was partially supported by JSPS KAKENHI (Grant Number 16K05185).

%
%
%

\end{document}